\documentclass{amsart}
\usepackage[utf8]{inputenc}   
\usepackage[T1]{fontenc}      % Police contenant les caractères français
\usepackage[all]{xy}
\usepackage{amssymb}
\usepackage{amsmath}
\usepackage{amsfonts}
\usepackage{mathrsfs}
\usepackage{varioref}
\usepackage{amsthm}
\usepackage{enumitem}
\usepackage{array}
\usepackage{stmaryrd}
\usepackage{xcolor}
\usepackage{tikz}
\usepackage[colorlinks=true]{hyperref}
\usepackage{longtable}
\usepackage{tabularx}
\allowdisplaybreaks

\hypersetup{
    colorlinks,
    linkcolor={red!50!black},
    citecolor={blue!50!black},
    urlcolor={blue!80!black}
}

\newtheorem{The}{Theorem}[section]
\newtheorem*{The*}{Theorem}
\newtheorem*{Con}{Conjecture}
\newtheorem{Pro}[The]{Proposition}
\newtheorem{Cor}[The]{Corollary}
\newtheorem{Lem}[The]{Lemma}

\theoremstyle{definition}
\newtheorem{Def}[The]{Definition}

\theoremstyle{remark}
\newtheorem{Rk}[The]{Remark}

\newcommand{\F}{\mathbb{F}}

\newcommand{\Z}{\mathbb{Z}}
\newcommand{\Q}{\mathbb{Q}}
\newcommand{\M}{\mathbf{L}}
\newcommand{\G}{\mathbf{G}}
\newcommand{\T}{\mathbf{T}}
\newcommand{\U}{\mathbf{U}}
\newcommand{\Irr}{\mathrm{Irr}}
\newcommand{\overbar}[1]{\mkern 1.5mu\overline{\mkern-1.5mu#1\mkern-1.5mu}\mkern 1.5mu}

\begin{document}

\title{Basic sets for unipotent blocks of finite reductive groups in bad characteristic}          
\author{Reda Chaneb}

\address{Universit\'e Paris Diderot, UFR de Math\'ematiques,
B\^atiment Sophie Germain, 5 rue Thomas Mann, 75205 Paris CEDEX 13, France.}
\email{reda.chaneb@imj-prg.fr}

\date{\today}                      
%\sloppy                       % Ne pas faire déborder les lignes dans la marge

\maketitle

\begin{abstract}
  We show the existence of a unitriangular basic set for unipotent blocks of simple reductive groups of classical type in bad characteristic with some exceptions. Then, we introduce an algorithm to count unipotent irreducible Brauer characters and we verify that this algorithm is effective for exceptional adjoint groups.
\end{abstract}

\section*{Introduction}
 Let $\G$ be a connected reductive group defined over $\F_q$ where $q$ is a power of a prime number $p$, and $F$ be the corresponding Frobenius endomorphism. We will denote by $G = \G^F$ the finite group of fixed points of $\G$ under $F$. It is a \emph{finite reductive group}. We are concerned with the representation theory of $G$ in \emph{tranverse characteristic} and in particular, we are interested in finding basic sets for certain blocks of $G$ called \emph{unipotent}. More precisely let us fix a prime number $\ell$ different from $p$. %We are interested in finding (unitriangular) basic sets for unipotent $\ell$-blocks of finite reductive groups. 
It is to be expected that one can find a natural basic set of characters for unipotent blocks of finite reductive groups, such that the decomposition matrix has unitriangular shape in this basic set. The behaviour of such basic sets depends on whether $\ell$ is a good or a bad prime number for $\G$. Let us recall the previous results on these basic sets.

\smallskip

 When $\ell$ is good and does not divide the order of the component group of the center of $\G$, the unipotent characters form a basic set for the unipotent blocks. This was first proved for $\mathrm{GL}_n(q)$ by Dipper  \cite{dipper_decomposition_1985-1} and $\mathrm{GU}_n(q)$ by Geck \cite{geck_decomposition_1991} using the so called Generalized Gelfand-Graev Representations (GGGRs). In both cases the basic set was shown to be unitriangular. The case of a general finite reductive group was settled by Geck-Hiss \cite{geck_basic_1991} and Geck \cite{geck_basic_1993}, but the question whether the decomposition matrix has unitriangular shape remains open in general.

\smallskip When $\ell$ is a bad prime number, much less is known: 
 for  classical groups with connected center and $\ell=2$ the existence of a unitriangular basic set for the unipotent
  blocks was shown in \cite{geck_basic_1994}, using again GGGRs. More recently,
  Kleshchev-Tiep found a unitriangular basic set for $\mathrm{SL}_n(q)$
  (\cite{kleshchev_representations_2009}) by studying the restrictions
  of characters of $\mathrm{GL}_n(q)$. Using this method and the results of Geck
  for $\mathrm{GU}_n(q)$, Denoncin generalized this result to $\mathrm{SU}_n(q)$
  (\cite{denoncin_stable_2017}).

%In \cite{dipper_decomposition_1985} and
%  \cite{dipper_decomposition_1985-1}, Dipper showed that unipotent
%  characters form a unitriangular basic set for the unipotent blocks
%  of $GL_n(q)$.
%  \item This result was generalized to $GU_n(q)$ by Geck in
%  \cite{geck_decomposition_1991} by using the so called Generalized Gelfand
%  Graev Representations (GGGRs).
%  \item In \cite{geck_basic_1991} and
%  \cite{geck_basic_1993}, Geck and Hiss gave a basic set when $\ell$
%  verify some conditions (see Theorem \ref{sec:GHbasicset}).
%  \item For  classical groups with connected center and $\ell=2$ and by using
%  GGGRs, the existence of a unitriangular basic set for the unipotent
%  block has been shown in \cite{geck_basic_1994}.
%  \item More recently,
%  Kleshchev and Tiep found a unitriangular basic set for $SL_n(q)$
%  (\cite{kleshchev_representations_2009}) by studying the restrictions
%  of characters of $GL_n(q)$.
%  \item Using this method and the results of Geck
%  for $GUn(q)$, Denoncin generalized this result to $SU_n(q)$
%  (\cite{denoncin_stable_2017}).

\smallskip

Following Geck--Hiss \cite{geck_basic_1991} and Geck \cite{geck_basic_1993,geck_basic_1994}, the strategy for constructing a unitriangular basic set consists in counting the number of irreducible unipotent Brauer characters, and then finding enough projective modules which will satisfy some unitriangularity condition. In this paper,  we achieve this in the case of classical groups (with the exception of the spin and half-spin groups), see Theorem \ref{thm:main}.

%
%The layout of this paper is as follows. In chapter 1, we recall some generalities about basic sets and introduce the results of Geck and Hiss. In chapter 2, we extend some results of Geck which provides a basic set for the unipotents block when $G$ is a classical group and $\ell=2$, more precisely we show this result.

\begin{The*}
  Assume that $p$ is an odd prime number, $\ell=2$ and let $\G$ be a simple group of type $B$, $C$ or $D$ except spin of half-spin. Then there exists a unitriangular basic set for the unipotent blocks of $G$.
\end{The*}

In the case of the exceptional groups, we provide a new algorithm to count the number of unipotent Brauer characters which we conjecture to be valid in general. We verify the output of our method in the following cases:
%
%In the last chapter, we discuss about a method to count number of irreducible modular characters in the unipotent blocks when $\ell$ is bad and $p$ is good for the three following cases:

\begin{itemize}
\item $\G$ is simple of type $B, C$ or $D$  
\item $\G$ is an exceptional group of adjoint type
\item $\G$ is $\mathrm{SL}_n(\overbar{\mathbb{F}}_p)$ (and so $G$ is either $\mathrm{SL}_n(q)$ or $\mathrm{SU}_n(q)$)
\end{itemize}
The advantage of our counting argument is that it generalizes Lusztig's parametrisation of unipotent characters, and hence is very close to a parametrisation of the unipotent Brauer characters. It is also compatible with Kawanaka's modified version of the GGGRs (\cite{noriaki_kawanaka_shintani_1987}). We believe that these representations could be used to produce a natural unitriangular basic set. 

\smallskip

The paper is organised as follows: in the first section we recall the results of Geck--Hiss and Geck on the basic sets for groups of Lie type. Section 2 is devoted to the case of classical groups in characteristic 2, where we prove our main theorem.  In Section 3 we explain how to extend Lusztig's parametrisation of unipotent characters to count unipotent Brauer characters. Finally, we provide tables for exceptional groups of adjoint type where our counting argument can be shown to give the expected result. 

\section{Background}
\subsection{Basic sets}

Let $G$ be a finite group. We fix an $\ell$-modular system $(K, R, k)$ for $G$, which consists of:

\begin{itemize}
\item a discrete valuation ring $R$ with unique maximal ideal $\mathfrak{m}$,
\item the fraction field $K:=\mathrm{Frac}(R)$, which we assume to have characteristic~$0$ and to be big enough for $G$ (\emph{i.e.} $K$ contains $|G|$-roots of unity),
\item the residue field $k:=R/ \mathfrak{m}$, of characteristic $\ell$.
\end{itemize}

If $\Lambda$ is either $K$ or $k$, we denote by $R_{\Lambda}(G)$ the Grothendieck group of the category of finitely generated $\Lambda G$-modules and by $\Irr_{\Lambda}(G)$ the set of isomorphism classes of irreducible $\Lambda G$-modules (it is a $\Z$-basis of $R_{\Lambda}(G)$). We will denote by $d: R_K(G) \to R_k(G)$ the decomposition map. By convention, we will call the \emph{decomposition matrix} of $G$ the transpose of the matrix of $d$ for the bases $\Irr_K(G)$ and $\Irr_k(G)$.

\begin{Rk}
Throughout this paper we may identify the irreducible $\Lambda G$-modules with their characters.
\end{Rk}

\smallskip

Recall that there is a unique decomposition $RG=\bigoplus_{i=1}^nB_i$ into indecomposable subalgebras (called \emph{blocks} of $G$). This decomposition induces a partition
$$\Irr_{\Lambda}(G)=\bigsqcup \limits_{i=1}^n \Irr_{\Lambda}(B_i)$$
of the set of irreducible representations (over $K$ or $k$). We will still call blocks of $G$ the sets $\Irr_{\Lambda}(B_i)$. The following definition is taken from \cite[14.3]{cabanes_representation_2004}.

\begin{Def}
  Let $B$ be a union of blocks and $R_k(B)$ be the subgroup of $R_k(G)$ generated by $\Irr_k(B)$. A \emph{basic set of characters} is a subset $\mathcal{B} \subset \Irr_K(B)$ such that $d(\mathcal{B})$ is a $\Z$-basis of $R_k(B)$. 
  
  Let $D$ be the part of the decomposition matrix of $B$ with rows labelled by $\mathcal{B}$. We say that  $\mathcal{B}$ is \emph{unitriangular} if, up to permutations of rows and columns, $D$ has lower unitriangular shape.  
\end{Def}

\begin{Rk}
  A unitriangular basic set provides a natural parametrization of Irr$_k(B)$ by elements of Irr$_K(B)$.
\end{Rk}

% \vspace*{0.5cm}

 In \cite{geck_basic_1991}, Geck and Hiss proved the existence of a basic set of characters in the case where $\ell$ is good and does not divide the order of the fixed point under $F$ of the component group of the center. This was generalised by Geck in the case where $\ell=2$ and $\G$ is a classical group with connected center. In addition, the basic set was shown to be unitriangular in that special case (see \cite{geck_basic_1994}). In the next section we  show that the result still holds when we drop the assumption that the center of $\G$ is connected, with the possible exception of the spin and half-spin groups.

\subsection{Blocks for finite groups of Lie type}

Let us recall some results about modular representation theory for finite reductive groups: let $\G^*$ be a dual group of $\G$. We still denote by $F$ the corresponding Frobenius endomorphism of $\G^*$. We set $G^*:=\G^{*F}$. By \cite{lusztig_irreducible_1977} (see also \cite{digne_representations_1991}), there is a partition of the set of irreducible characters into so-called rational Lusztig series:
$$\Irr_K(G)=\bigsqcup \limits_{s}\mathcal{E}(G,s)$$
where $s$ runs over representatives of $G^*$-conjugacy classes of semisimple elements of $G^*$. The partition of $\Irr_K(G)$ into Lusztig series behaves particularly well with respect to the partition into blocks. 

\begin{The}[Brou\'e--Michel \cite{broue_bloc_1989}]\label{sec:BM-blocks} 
  Let  $s \in G^*$ be a semisimple element of order prime to $\ell$ and set
$$B_s(G):=\bigcup \limits_t \mathcal{E}(G^F, st)$$
where $t$ runs over the set of semisimple $\ell$-elements of $C_{G^{*F}}(s)$. Then $B_s(G)$ is a union of blocks.
\end{The}

%\vspace*{0.5cm}

When there is no ambiguity on the underlying group, we will denote $B_s(G)$ by $B_s$. The blocks contained in $B_1$ will be called \emph{unipotent blocks}. Under some restrictions on $\ell$, Geck and Hiss proved that there is a natural basic set for $B_s$.

\begin{The}[Geck--Hiss \cite{geck_basic_1991}]\label{sec:GHbasicset}
  Assume that $\ell$ is good for $G$ and does not divide the order of $(Z_{\G}/Z_{\G}^{\circ})^F$. Let $s \in G^*$ be a semisimple element of order prime to $\ell$. Then $\mathcal{E}(G,s)$ is a basic set for $B_s$.
\end{The}

\begin{Rk}\label{bad-examples}
  If we remove the assumptions on $\ell$ in Theorem \ref{sec:GHbasicset}, $\mathcal{E}(G,s)$ is no longer a basic set:
  \begin{itemize}
  \item If $G=\mathrm{SL}_2(3)$ and $\ell=2$, then $\ell$ divides
    $|(Z_{\G}/Z_{\G}^\circ)^F|$. The group $G$ has 2 unipotent characters but
    $\Irr_k(B_1)$ contains 3 elements (see \cite{bonnafe_representations_2011}).
  \item If $G=G_2(q)$ and $\ell=2$, then $\ell$ is bad for $G$. Here, $d(\mathcal{E}(G,1))$ generates $R_{k}(B_1)$ but is not a basic set (see \cite{geck_basic_1991}).
\end{itemize}
\end{Rk}

%\vspace*{0.5cm} 

The theorem above does not indicate whether the basic set is unitriangular or not. However, it was conjectured by Geck to be the case, at least for unipotent blocks.

\begin{Con}[Geck  {\cite[Conj. 3.4]{geck_modular_1997}}]
  Suppose that $\ell$ is good for $G$ and does not divide the order of $(Z_{\G}/Z_{\G}^{\circ})^F$. Then the unipotent characters form a unitriangular basic set of $B_1$.
\end{Con}

% We denote by $m_s$, the number of irreducible Brauer characters of $B_s$ (If we need to indicate the underlying group, we will use the notation $B_s(G)$ and $m_s(G)$). The blocks contained in $B_1$ will be called the unipotent blocks.
 \vspace*{0.5cm}

 \section{A basic set for classical groups in characteristic 2}

Throughout this section we will assume that $\G$ has simple components of type $B$, $C$ $D$, that $p$ is odd and that $\ell=2$.

 \subsection{Finding a basic set}
Under the condition that the center of $\G$ is connected, Geck showed in \cite{geck_basic_1994} the existence of a unitriangular basic set for the unipotent blocks. We will generalise this result to some cases where the center is no longer connected by using the same methods. We start off with some results which are valids for aribtrary finite groups.

 %We have seen in the previous part that for $\ell=2$ the set of unipotent characters is not necessarly a basic set, so in order to find a basic set, we will use this proposition:

\begin{Pro}\label{sec:existence-basic-set}
Let $B$ be a union of blocks and $n:=|\Irr_k(B)|$. Assume that there exist ordinary characters $\chi_1, \ldots, \chi_n$ and projective characters $\Phi_1, \ldots, \Phi_n$ in $B$ such that the matrix of scalar products $\big(\langle\chi_i, \Phi_j\rangle\big)_{1 \leq i,j \leq n}$ has lower unitriangular shape. Then the characters $\chi_1, \ldots, \chi_n$ form a unitriangular basic set for $B$.
\end{Pro}

%\vspace*{0.5cm}

\begin{proof}
Let $d: R_K(G) \longrightarrow R_k(G)$ be the decomposition map. Since $\Phi_j$ are projective characters, we have the well-known relation $\langle\chi_i, \Phi_j\rangle = \langle d(\chi_i), \Phi_j\rangle$. 
Let us denote by $\varphi_1, \ldots, \varphi_n$ the irreducible Brauer characters lying in $B$. For $ 1 \leq i \leq n$, we can write  $d(\chi_i)=\sum_{l=1}^n d_{\chi_i, \varphi_l} \varphi_l$. So we have
$$\langle\chi_i, \Phi_j\rangle =  \sum_{l=1}^n d_{\chi_i, \varphi_l}\langle\varphi_l, \Phi_j\rangle.$$ 
In terms of matrices, this equality is:
$$(d_{\chi_i, \varphi_j})_{1 \leq i,j \leq n} \times (\langle\varphi_i, \Phi_j\rangle)_{1 \leq i, j \leq n} = (\langle\chi_i, \Phi_j\rangle)_{1 \leq i,j \leq n}.$$

    We want to show that $(d_{\chi_i, \varphi_k})$ is lower unitriangular, using the fact that the matrix on the right side of the equality is lower unitriangular. The result follows from the lemma below.
  \end{proof}

  \vspace*{0.5cm}

  \begin{Lem}
Let $A$, $B$ and $C$ be three square matrices of size $n \times n$ with non-negative integer coefficients such that $AB=C$ and $C$ is lower unitriangular. Then, up to permutations of columns, $A$ is unitriangular as well.
\end{Lem}

\begin{proof}
  Let us write $A=(a_{i,j})_{1 \leq i,j \leq n}$, $B=(b_{i,j})_{1 \leq i,j \leq n}$, $C=(c_{i,j})_{1 \leq i,j \leq n}$. The matrix $C$ is lower unitriangular so 
  \begin{align}
    0&=c_{i,j}= \sum \limits_{k=1}^n a_{i,k}b_{k,j} \quad \forall 1 \leq i<j \leq n.\\
    1&=c_{i,i}= \sum \limits_{k=1}^na_{i,k}b_{k,i}  \quad \forall 1 \leq i \leq n.
    \end{align}
    All the coefficients are non-negative so that implies:
    \begin{itemize}
    \item $a_{i,k}b_{k,j}=0$ for all $1 \leq k \leq n$ and $1 \leq i<j \leq n$;
     \item Given $1 \leq i \leq n$, there is a unique integer $k$ with $1\leq k \leq n$ such that $a_{i,k}=b_{k,i}=1$. If $k' \neq k$, then $a_{i,k'}=0$ or $b_{k',i}=0$.
     \end{itemize}
     Therefore, in every column of $A$, there is a coefficient whose value is 1 and such that every coefficient above is 0. Similarly, in each row of $B$ there is a coefficient whose value is $1$ and such that  every following coefficient is zero. Every coefficient of the first row of $A$ (resp. the first column of $B$) should be either $0$ or $1$ and, since $\det(A)$ (resp. $\det(B)$) is non-zero, there is at least one coefficient whose value is $1$. Let us show that the first row of $A$ has a unique coefficient whose value is $1$: let $k_1,\ldots,k_s$ be the indices such that $a_{1,k_i}=1$ and suppose $s\geq2$. By (1), when $1 \leq i \leq s$, and $2 \leq j \leq n$, $0=a_{1,k_i}b_{k_i,j}=b_{k_i,j}$ so every coefficient the $k_i$th line of $B$ is $0$ except the first one, whose value is $1$. So $c_{1,1}= \sum_{i=1}^s a_{1,k_i}b_{k_i,1}=s$, which gives a contradiction. The same argument applied to every row of $A$ shows that, up to permutation of columns, $A$ is lower unitriangular.
   \end{proof}

%   \vspace*{0.5cm}

Consequently, in order to find a unitriangular basic set of characters for the unipotent blocks it is enough to:
   \begin{itemize}
   \item Count the number of elements of $\Irr_k(B_1)$, which is the content of the section below.
   \item Find projective and ordinary characters satisfying the assumptions of Proposition \ref{sec:existence-basic-set}. This can be done by using the theory of Generalized Gelfand-Graev characters and results of Geck, H\'ezard and Taylor (\cite{geck_unipotent_2008}, \cite{taylor_unipotent_2013}).
   \end{itemize}

\subsection{Counting the unipotent Brauer characters}

From now on, we will denote by $m_s$ the number of elements in $\Irr_k(B_s)$, or $m_s(G)$ if we need to specify the underlying group $G$. In \cite{geck_basic_1994}, Geck showed that when $\G$ has simple components of classical type, $m_1(G)$ is equal to the number of unipotent classes of $G$, provided that the center of $\G$ is connected. The aim of this part is to show that the connectedness condition on the center can be removed. If $\G$ has simple components of type $B$, $C$, $D$ then we first recall that the centraliser in $\G$ of a semisimple element of odd order is a rational Levi subgroup of $\G$.

\begin{Lem}
Let $\G$ be a reductive group such that $Z(\G)/Z(\G)^\circ$ is a $2$-group.

\begin{enumerate}
\item[$\mathrm{(1)}$] For any Levi subgroup $\M$ of $\G$, $Z(\M)/Z(\M)^\circ$ is a $2$-group.
\item[$\mathrm{(2)}$] If $s$ is a semi-simple element of odd order of $\G^*$, then $C_{\G^*}(s)$ is connected. If moreover the order of $s$ is divisible by good primes only then $C_{\G^*}(s)$ is a rational Levi subgroup of $\G$.
\end{enumerate}
\end{Lem}

\begin{proof}
%  \vspace{0.1cm}
Let $X$ be the group of characters of a rational torus $\T$ contained in $\M$, $\Phi$ be the corresponding set of roots of $\G$ and $\Phi_{\M} \subset \Phi$ be the set of roots of $\M$. Given $A$ a subgroup of $X$ and $\mathbf{S}$ a subtorus of $\T$ we define:
$$\begin{aligned}
 A^{\perp}& \, :=\{t \in \T \, \mid\, \chi (t)=0 \ \forall \chi \in A\} \\
\mathbf{S}^{\perp}& \, :=\{\chi \in X \, \mid\, \chi (t)=0 \ \forall t \in \mathbf{S} \}.
\end{aligned}$$
    
Note that $A \subset A^{\perp \perp}$ but there is no equality in general. The following properties can be found for example in  \cite[Prop 0.24, Lem. 13.14 and Rem. 13.15]{digne_representations_1991}.

 \begin{enumerate}[label=(\alph*)]
\item $Z(\G)/Z(\G)^{\circ}$ is isomorphic to the torsion group of $X/\Z \Phi^{\perp \perp}$.
\item $A^{\perp \perp}/A$ is the $p$-torsion subgroup of $X/A$.
\item If $(\G^*,\T^*)$ is dual to $(\G, \T)$, then for any $s \in \T^*$ the group of $C_{\G^*}(s)/C_{\G^*}(s)^\circ$ is isomorphic to a subgroup of $Z(\G)/Z(\G)^{\circ}$.
\item The exponent of $C_{\G^*}(s)/C_{\G^*}(s)^{\circ}$ divides the order of $s$.
\end{enumerate}
\smallskip

By (a), the group $Z(\G)/Z(\G)^{\circ}$ is a $2$-group if and only if $X /\Z \Phi ^{\perp \perp}$ has only $2$-torsion. This is equivalent to $X/\Z \Phi$ having only $2$ and $p$-torsion by remarking that  
  $$X /Z \Phi^{\perp \perp} \simeq (X/ \Z \Phi) / (\Z \Phi^{\perp \perp}/\Z \Phi)$$  
  and then applying (b) with $A=\Z \Phi$. The fact that $\Z \Phi_{\M}$ is a direct summand of $\Z \Phi$ implies that $X/\Z \Phi_{\M}$ has only $2$-torsion and $p$-torsion. Indeed, suppose $X/\Z \Phi_{\M}$ has $m$-torsion where $m$ is an integer prime to $2$ and $p$. Then there exists an element $\chi \in X$ such that $m \chi \in \Z \Phi_{\M}$, but  that implies that $\chi \in \Z \Phi$ because $X/\Z \Phi$ has only $2$ and $p$-torsion. The fact that $\chi \in \Z \Phi$ and $m \chi \in \Z \Phi_{\M}$  forces $\chi \in \Z \Phi_L$ since $\Z \Phi_{\M}$ is a direct summand of the free $\Z$-module $\Z \Phi$. Therefore $X /\Z \Phi_{\M}$ has no $m$-torsion. By (b), this is equivalent to $Z(\M)/Z(\M)^\circ$ being a $2$-group. This proves the assertion~(1). 

\smallskip
By (c), $C_{\G^*}(s)/C_{\G^*}(s)^{\circ}$ is isomorphic to a subgroup of $Z(\G)/Z(\G)^{\circ}$, hence is a $2$-group. But the exponent of this group divides the order of $s$ by (d) so $C_{\G^*}(s)/C_{\G^*}(s)^{\circ}$ is trivial and $C_{\G^*}(s)$ is connected. Finally, the connectedness and   \cite[Prop. 2.1]{geck_basic_1991} implies that $C_{\G^*}(s)$ is a rational Levi subgroup of $\G^*$ which gives (2). 
\end{proof}

An immediate consequence is:

\begin{Cor}\label{cor:levi}
If $\G$ has simple components of type $B$, $C$ or $D$ and $s$ is a semisimple element of odd order of $G$ then $C_{\G}(s)$ is a rational Levi subgroup of $\G$.
\end{Cor}

%\vspace*{0.5cm}

Let $S_{2'}(G)$ be a set of representatives of $G$-conjugacy classes of semisimple elements of odd order of $G$. The result in \cite[Prop. 4.2]{geck_basic_1991} gives, under some conditions, a bijection between $S_{2'}(G)$ and $S_{2'}(G^*)$, which is one of the main ingredients to count irreducible Brauer characters in unipotent blocks. We need a slight modification of the statement to be able to get this bijection in case the center is not connected. If $s \in S_{2'}(G)$, we denote by $C_{\G}(s)^*$ a Levi subgroup of $\G^*$ dual to $C_{\G}(s)$ (we recall that by the lemma above and under the assumptions on $\G$, $C_{\G}(s)$ is a rational Levi subgroup of $\G$). 

\vspace*{0.3cm}

\begin{Pro}\label{sec:bijoddorder}
 Assume that $\G$ is a Levi subgroup of a simple group of type $B, C$ or $D$. Then there is a bijection
 $$S_{2'}(G) \to S_{2'}(G^*) \quad t \mapsto t' $$
 such that we have an isomorphism $C_{\G}(t)^* \simeq C_{\G^*}(t')$ defined over $\F_q$.
\end{Pro}

%\vspace*{0.3cm}

\begin{proof}
Let $\T$ be a maximally split torus of $\G$ contained in an $F$-stable  Borel subgroup $\mathbf{B}$. Let $\T^*$ be a dual torus. We denote by $\Delta$ the basis of the root system associated to the pair $(\T,\mathbf{B})$. The Weyl groups of $\T$ and $\T^*$ will be denoted by $W$ and $W^*$ respectively. From the pairing between characters of $\T$ and $\T^*$ we have a group isomorphism $W \simeq W^*$. We will write $w^*$ for the image of $w$ under this isomorphism. By assumption, the torus $\T$ is contained in a simple group of type $B$, $C$ or $D$, in particular the determinant of the Cartan matrix of this group is a power of 2 so the proof of Proposition 4.2 in \cite{geck_basic_1991} gives an $F$-equivariant bijective map which intertwines the action of $W$ and $W^*$:
  $$ \varphi: \quad \T_{2'} \to \T^*_{2'}  $$
with the following property: if $\alpha$ is a root, $\alpha^{\vee}$ the corresponding coroot and $t \in T_{2'}$ then
\begin{equation}\alpha(t)=0 \text{ if and only if } \alpha^{\vee}(\varphi(t))=0. \tag{P}\end{equation}
 
Let $s \in G$ be a semisimple element of odd order, then $\M:=C_{\G}(s)$ is a Levi subgroup of $\G$ by
\ref{cor:levi}. Therefore there exists $g \in \G$ such that $\M^g=\M_{\Delta_1}$ is the standard Levi subgroup corresponding to a subset $\Delta_1 \subset \Delta$. Up to multiplying $g$ by an element of $\M_{\Delta_1}$ on the right we can assume without loss of generality that  $n_w=g^{-1}F(g) \in N_{\G}(\T)$. We will denote by $w$ its image in $W$. If we set $t=s^g$, then $t \in T_{2'}$ and $C_{\G}(t)=\M^g = \M_{\Delta_1}$. Let $t'= \varphi (t) \in T^*_{2'}$. Let $g' \in \G^*$ such that $g'^{-1}F(g') \in N_{\G^*}(\T^*)$ and has image $w^*$ in $W^*$. Then
$$ F(t')=n_{w^*}^{-1}t'n_{w^*}$$
and $s'=g't'g'^{-1}$ is $F$-stable. We set $\M' = C_{\G^*}(s')$.  Let us recall the following properties:
\begin{itemize}
  \item $\varphi$ is $(\G, \G^*)$-equivariant under the action by conjugation (see the proof of Theorem 4.1 of \cite{geck_basic_1991});
  \item  $C_{\G}(s)$ is connected so every $F$-stable element conjugate to $s$ under $\G$ is conjugate to $s$ under $G$ (see for example \cite[Prop. 3.25]{digne_representations_1991}). Note that $C_{\G^*}(s')$ is also connected so the same hold for $s'$;
   \item Two elements of odd order of $T$ are conjugate under $G$ if and only if they are conjugate under $W$ (see \cite[0.12 (iv)]{digne_representations_1991}).
  \end{itemize}

 Consequently, the map $s \mapsto s'$ induces a map $\Phi: S_{2'}(G) \to S_{2'}(G^*)$. Let us prove the injectivity of $\Phi$: suppose that there are semisimple elements of odd order of $G$, say $s_1$ and $s_2$ such that $s_1':=\Phi(s_1)$ and $s_2':=\Phi(s_2)$ are conjugate under $G^*$. If we write $s_1'=g_1't_1'g_1'^{-1}$ and $s_2'=g_2't_2'g_2'^{-1}$ as above, then $t_1'$ and $t_2'$ are conjugate under $G$ so they are conjugate under $W^*$. Hence, because of the equivariance properties of $\varphi$, $t_1:=\varphi^{-1}(t_1')$ and $t_2:=\varphi^{-1}(t_2')$ are conjugate under $W$, hence $s_1$ and $s_2$ are conjugate under $G$. The surjectivity of $\varphi$ implies immediately the surjectivity of $\Phi$ so $\Phi$ is bijective.

\smallskip

  Finally we have to prove that there is an isomorphism between $\M$ and $\M'$ which is defined over $\F_q$. Given $\alpha$ a root, let $U_{\alpha}$ be the corresponding root subgroup of $G$. We have:
$$\begin{aligned}
    {}^g\M'=C_{\G}(t') & = \langle\T^*, \U_{\alpha^{\vee}}\, \mid \, \alpha^{\vee}(t')=1\rangle\\
    & =\langle\T^*, \U_{\alpha^{\vee}} \, \mid \, \alpha^{\vee}(\varphi^{-1}(t))=1\rangle\\
                   &=\langle\T^*, \U_{\alpha^{\vee}} \, \mid \, \alpha(t)=1\rangle & \\
    &=\M_{\Delta_1^{\vee}}
  \end{aligned}$$
  The third equality comes from (P). Since $g'^{-1}F(g')$ represents $w^*$, the $G$-conjugacy class of $\M$ corresponds naturally to the $G^*$-conjugacy class of $\M'$ and so by \cite[\S 7.3]{lusztig_irreducible_1977}, there is an isomorphism between $\M^*$ and $\M'$ defined over $\F_q$. 
\end{proof}

%\vspace*{0.5cm}

Let $s \in S_{2'}(G^*)$. To the Levi subgroup $\M^*_s:=C_{\G^*}(s)$  of $\G^*$ we attach $\M_s \subset \G$, a rational Levi subgroup of $\G$ dual to $\M^*_s$. In \cite[11.8]{bonnafe_categories_2003}, Bonnafe and Rouquier have shown the existence of a Morita equivalence between $B_1(L_s)$ and $B_s(G)$. In particular $ m_s(G)=m_1(L_s)$ and the decomposition matrices are the same. This means that we can reduce it to the case of a unipotent block of a Levi subgroup. Let us prove the main result of this section.

%\vspace*{0.3cm}

\begin{Pro}\label{sec:countingirrkb1}
Let $\G$ be a Levi subroup of a simple group of type $B,C$ or $D$ in odd characteristic. Then $m_1(G)$ is equal to the number of unipotent classes of $G$.
\end{Pro}

\begin{proof}
  We will proceed by induction on the dimension: let $s \in S_{2'}(G^*)$ and $\M_s$ be a rational Levi subgroup of $\G$ dual to the Levi subgroup  $C_{\G^*}(s)$ of $\G^*$. We denote by $u(L_s)$ the number of unipotent classes of $L_s$. If $s$ in non-central, then $\dim \M_s <  \dim \G$, so
  $$u(L_s)=m_1(L_s)=m_s(G).$$
The first equality comes from the induction hypothesis, the second one from the Morita equivalence between $B_1(L_s)$ and $B_s(G)$ as said above. Let $a$ be the number of $2$-regular classes of $G$. Using the fact that $a$ is also the number of elements in $\Irr_k(G)$, we have
$$\begin{aligned}
    a & =\sum_{s \in S_{2'}(G^*)} m_s(G)\\
        &=\sum_{\underset{s\notin Z(G^*)}{s \in S_{2'}(G^*)}} u(L_s) + \sum \limits_{s\in Z(G^*)_{2'}} m_s(G)\\
    & =\sum_{\underset{s\notin Z(G^*)^F}{s \in S_{2'}(G^*)}} u(L_s)  + |Z(G^*)_{2'}| \cdot m_1(G)    
 \end{aligned}$$
The last equality comes from the fact that for any central element $s$  of odd order of $G^*$, tensoring by the linear character of $C_{\G^*}(s)$ attached to $s$ as in \cite[13.30]{digne_representations_1991} provides a natural isomorphism between $B_1(G)$ and $B_s(G)$, so $m_1(G)=m_s(G)$. By using the fact that $a$ is the number of $2$-regular classes of $G$ we also have:
$$\begin{aligned} a = &\, \sum_{\underset{s\notin Z(G)}{s \in S_{2'}(G)}} u(C_G(s)) + |Z(G)_{2'}| \cdot u(G)  \\
= & \, \sum_{\underset{s\notin Z(G)}{s \in S_{2'}(G)}} u(L_{s'}) + |Z(G)_{2'}| \cdot u(G)
\end{aligned}$$
where the second equality comes from the fact that the isomorphism $C_\G(s)^* \simeq C_{\G^*}(s')$ is defined over $\F_q$. 
\smallskip

To conclude it remains to show that $|Z(G)_{2'}|=|Z(G^*)_{2'}|$. This comes from the fact that the bijection  $S_{2'}(G^*) \simeq S_{2'}(G)$ of \ref{sec:bijoddorder} induces a bijection between central elements of odd order of $G$ and $G^*$.\end{proof}

\subsection{Existence of the basic set}

From now on, we assume that $\G$ is a simple group of type $B, C$ or $D$ in odd characteristic. Recall that our aim is to find a unitriangular basic set of characters for the unipotent blocks of $G$. The problem of finding a unitriangular basic set for the unipotent blocks reduces to finding $m_1(G)$ projective characters and ordinary characters which satisfy the assumptions of Proposition \ref{sec:existence-basic-set}. Moreover, according to Proposition \ref{sec:countingirrkb1}, $m_1(G)$ is the number of unipotent classes of $G$. We have an interesting candidate: Kawanaka constructed a family of projective characters parametrised by the unipotent classes of $G$. More precisely, let $C$ be an $F$-stable unipotent class of $\G$, let $u_1,\dots,u_r$ be representatives of the $G$-conjugacy classes of $C^F$. To each $u_i$, we can associate a projective character $\gamma_{u_i}$, called the generalized Gelfand Graev character (for more details on the construction see for example \cite{taylor_unipotent_2012}).

%\vspace*{0.5cm}
\smallskip

Following \cite{lusztig_unipotent_1992}, we can associate to $\rho \in$ Irr$_K(G)$ a unique $F$-stable unipotent class $C_{\rho}$ of $\G$ satisfying the two following conditions:

\begin{enumerate}
\item $\sum_{x\in C_{\rho}^F} \rho(x) \neq 0$,
\item $C_{\rho}$ is of maximal dimension for the condition above.
\end{enumerate}

We say that $C_\rho$ is the \emph{unipotent support} of $\rho$. By using the GGGRs, we can also associate to $\rho$ another $F$-stable unipotent class $C_{\rho}^*$ of  $\G$, the \emph{wave front set} of $\rho$, which is the unique unipotent class satisfying the following conditions:

\begin{enumerate}
\item $(\rho, \gamma_u)\neq 0$ for some $u \in C_{\rho}^{*F}$,
\item$(\rho, \gamma_v) \neq 0$ implies $\dim C_v \leq  \dim C_{\rho}^*$,
\end{enumerate}
where $C_v$ is the unipotent class of $\G$ containing $v$. Actually, the unipotent support and the wave front set are closely related. More precisely, let $D_G$ be the Alvis-Curtis duality for representations of $G$ (see the definition in \cite[8.8]{digne_representations_1991}). It is known that $D_G$ maps an irreducible character to an irreducible character up to a sign. If $\rho \in \Irr_K(G)$, we denote by $\rho^*$ the irreducible character of $G$ such that $\rho^*= \pm D_G(\rho)$. We call $\rho^*$ the dual character of $\rho$. The relation between unipotent support and wave front set comes from this duality:
$$\forall \rho \in \Irr_K(G) \quad C^*_{\rho}=C_{\rho^*}.$$
Note that all these properties were first proved under the assumption that $p$ and $q$ were large enough. These were later generalised to the case where $p$ is good prime in \cite{taylor_generalized_2016}.

\smallskip

Let us fix some notation: given $x \in G$, we denote by $A_{\G}(x)$ the component group of $C_{\G}(x)$. Since this group depends only on the conjugacy class of $x$ in $\G$, we will use the notation $A_C$ where $C$ is the conjugacy class of $x$. In order to use Proposition \ref{sec:existence-basic-set}, we need to find suitable ordinary characters. Those characters are provided by this result of Geck, H\'ezard and Taylor.

\begin{Pro}[Geck--Hézard, Taylor \cite{geck_unipotent_2008,taylor_unipotent_2013}] \label{sec:existence-basic-set-2}
  Let $C$ be an $F$-stable unipotent class of $\G$ and $u_1, \ldots, u_r$ be representatives of the $G$-conjugacy classes of $C^F$. Assume that $A_C$ is abelian. Then there exist an element $s$ and characters $\chi_1,\ldots,\chi_r \in \mathcal{E}(G, s)$ such that
  \begin{itemize}
  \item[$\mathrm{(i)}$] the image of $s$ in the adjoint quotient $\G^* / Z(\G^*)$ is quasi-isolated;
  \item[$\mathrm{(ii)}$] $\chi_1, \ldots, \chi_r$ have unipotent support $C$;
  \item[$\mathrm{(iii)}$] $\big(\langle\chi_i^*,\gamma_{u_j}\rangle_G\big)_{i,j}$ is the identity matrix.
  \end{itemize}
\end{Pro}

%\vspace*{0.5cm}

We need some more notation: let $C_1, \ldots, C_r$ be the $F$-stable unipotent classes of $\G$ ordered such that $i<j$ whenever $\dim C_i < \dim C_j$. For each class $C_i$, let $u_{i,1},\ldots,u_{i,r_i}$ be representatives of the $G$-conjugacy classes of $C_i^F$. We denote by $\gamma_{i,u_i}$ the GGGR corresponding to $C_i$ and by $\chi_{i,u_i}$ the characters given by the proposition above. Now, we can state the main result.

\begin{The}\label{thm:main}
Assume that $p$ is odd and $\ell = 2$. Let $\G$ be a simple group of type $B$, $C$ or $D$, but not a spin or half-spin group. Then the set $(\chi_{i, u_j}^*)_{1 \leq i \leq n, 1 \leq j \leq r_i}$ form a unitriangular basic set of characters for the unipotent blocks of G.
\end{The}

%\vspace*{0.1cm}

\begin{proof}
 First, we need to show that the characters of the statement lie in a unipotent block. By definition of $B_1$, it is enough to show that they all belong to a Lusztig series associated to a $2$-element. By Proposition \ref{sec:existence-basic-set-2}, the characters $\chi_{i, u_j}$ all belongs to $\mathcal{E}(G,s)$ where the image $s'$ of $s$ in the adjoint quotient of $\G^*$ is a quasi-isolated element. But for adjoint groups of type $B$, $C$ or $D$, all quasi-isolated elements are $2$-elements (\cite{bonnafe_quasi-isolated_2005}). Hence, there is an integer $k$ such that $s^{2^k}$ belongs to $Z(\G^*)$ which is a $2$-group, so $s$ is a 2-element. This proves that  $\chi_{i,u_j}$ (and its Alvis-Curtis dual $\chi_{i,u_j}^*$, this is a direct consequence of \cite[12.8]{digne_representations_1991}) belongs to $B_1$.
 
 \smallskip
 
Now, we have $m_1$ ordinary and projective characters in $B_1$ so, according to Proposition \ref{sec:existence-basic-set}, we simply have to ensure that the matrix $(\chi_{i,u_j}^*, \gamma_{i',u_{j'}})_{G}$ is lower unitriangular. It is shown in \cite{liebeck_unipotent_2012} that if $C$ if a unipotent class of a simple group of type $B,C, D$ which is not a spin group nor a half-spin group, then $A_C$ is abelian. Therefore according to Proposition \ref{sec:existence-basic-set-2}, for  each class $C_i$, the scalar matrix $\big(\langle\chi_{i, u_j}, \gamma_{i, u_{j'}}\rangle_{G}\big)_{j,j'}$ is the identity matrix. It remains to check that $\langle\chi_{i, u_j}^*, \gamma_{i', u_{j'}}\rangle_{G}$ is $0$ whenever $i<i'$. This simply comes from the fact that $C_i$ is the unipotent support of $\chi_{i, u_j}$, so the wave front set of $\chi_{i, u_j}^*$.  
\end{proof}

\section{Counting Brauer characters}

From now on, we will assume that $\G$ is simple adjoint, $p$ is good for $\G$ and $\ell$ is different from $p$. It follows from the work of Lusztig that unipotent characters of finite reductive groups can be parametrised in term of special unipotent classes. The aim of this section is to introduce the notion of \textit{$\ell$-special} unipotent classes in order to count the number of irreducible Brauer characters lying in the unipotent blocks in terms of those classes. We believe that this should provide a natural parametrisation of the irreducible Brauer characters lying in unipotent blocks.

\subsection{On representations of Weyl groups and unipotent characters}

\subsubsection{Special representations and families}
We summarize briefly Lusztig's results on the classification of unipotent characters (see \cite[\S4]{lusztig_g._characters_1984} or \cite[\S11]{carter_finite_1985}). To $\phi \in \Irr_K(W)$ we can associate two polynomials with coefficients in $\Q$:
$$\begin{aligned}
P_{\phi}(X)= &\, \gamma_{\phi}X^{a_{\phi}}+ \cdots +\delta_{\phi}X^{b_{\phi}}, \quad a_\phi \leq b_\phi \\
\widetilde{P}_{\phi}(X)=& \,\widetilde{\gamma}_{\phi}X^{\widetilde{a}_{\phi}}+\cdots+\widetilde{\delta}_{\phi}X^{\widetilde{b}_{\phi}}, \quad \widetilde{a}_{\phi} \leq \widetilde{b}_{\phi}
\end{aligned}$$

The first polynomial $P_{\phi}$, the \emph{fake degree} of $\phi$, is defined by considering the multiplicities of $\phi$ in each component of a certain quotient of the graded algebra of the natural representation of $W$, whereas the polynomial $\widetilde{P}_{\phi}$, the \emph{generic degree of $\phi$}, is constructed by the theory of Hecke algebras. According to \cite[11.3.4]{carter_finite_1985},  $\widetilde{a}_{\phi} \leq a_{\phi}$ , we are interested in the case where we have an equality.

%\vspace*{0.3cm}

\begin{Def}[Lusztig]
 $\phi \in \Irr_K(W)$ is called \textit{special} if $a_\phi=\widetilde{a}_\phi$.
\end{Def}

The set $\Irr_K(W)$ can be partitioned into subsets called families and it turns out that each family of $\Irr_K(W)$ contains exactly one special character (see for example \cite[12.3]{carter_finite_1985}). Moreover, there also exists a partition of the set of unipotent characters of $G$ indexed by the $F$-stable families of $\Irr_K(W)$ (\cite[12.3]{carter_finite_1985}). From now on, we will use the term family, either for irreducible characters of $W$ or for unipotent characters of $G$. Given a $F$-stable family $\mathcal{F}$, we will use the notations $\mathcal{F}(G)$ or $\mathcal{F}(W)$ if we need to specify if we work with $G$ or $W$. To explain the parameterisation of unipotent characters by Lusztig, we will need a few notations:

\vspace*{0.3cm}

\begin{Def} If $\Gamma$ is a group on which $F$ acts, we denote by
    $\tilde{\Gamma}$ the semidirect product of $\Gamma$ with the
    cyclic group generated by $F$ where $F.x.F^{-1}=F(x)$ if
    $x \in \Gamma$.
    $\tilde{\mathcal{M}}(\Gamma)$ is the set whose
    elements are $\tilde{\Gamma}$-conjugacy classes of pairs $(x,\sigma)$ where
    $x \in \Gamma .F$ and
    $\sigma \in \Irr_K(C_{\tilde{\Gamma}}(\sigma))$.

\end{Def}

\begin{Rk}\label{sec:MGamma} The following observations will be useful later:
  \begin{itemize} 
  \item If $x:=a.F \in \Gamma.F$, the
    $\tilde{\Gamma}$-conjugates of $x$ consist of elements of the form
    $b.F$ where $b \in \Gamma$ is $F$-conjugate to $a$.
\item If $F$ acts trivially on $\Gamma$, $\tilde{\Gamma}$ is the direct product of $\Gamma$ with the cyclic group generated by $F$ and $\tilde{\mathcal{M}}(\Gamma)$ is in bijection with the set of $\Gamma$-conjugacy classes of pairs $(x, \sigma)$ where $x \in \Gamma$ and $\sigma \in \Irr_K(C_{\Gamma}(x))$.
\end{itemize}

\end{Rk}

To each $F$-stable family $\mathcal{F}$ we can associate a group $\mathcal{G}_{\mathcal{F}}$ on which $F$ acts (this is done case by case, see \cite[\S4]{lusztig_irreducible_1977}).  It has been proved by Lusztig that unipotent characters in the family $\mathcal{F}(G)$ are parametrized by $\tilde{\mathcal{M}}(\mathcal{G}_{\mathcal{F}})$ (see the tables in Chapter 13 of \cite{carter_finite_1985}). %{\bf[ C'est un peu plus complique en general mais il pour les groupes simples on peut toujours supposer que $F$ agit trivialement sur $\mathcal{G}_{\mathcal{F}}$ et ca marche]}

\subsubsection{Springer correspondence and special unipotent classes}

In \cite{springer_construction_1978}, Springer introduced a method to relate irreducible characters of $W$ and unipotent classes of $\G$. More precisely, to $u \in G$ a unipotent element and $\psi$ an irreducible character of $A(u)$, we can associate a $KW$-module $E_{u, \psi}$ which is either $0$ or irreducible, and $E_{u, \psi} \simeq E_{u', \psi'}$ if and only if $u$ is conjugate to $u'$ and $\psi=\psi'$. Springer showed that every irreducible representation of $W$ over $K$ is of the form $E_{u,\psi}$, and special representations have a specific form under this correspondence:

\begin{Pro}
  Any special irreducible representation of $W$ is of the form $E_{u,1}$ where $u$ is a unipotent element of $\G$.
\end{Pro}

The last proposition leads to the following definition:

\begin{Def}
Let $u$ be a unipotent element of $\G$. We say that $u$ is \textit{special} if $E_{u,1}$ is a special irreducible representation of $W$. In this case we will also say that the conjugacy class of $u$ is special. 
\end{Def}

%\vspace*{0.5cm}

Recall that each family contains exactly one special character, so there is a bijective correspondence between families of $\Irr_K(W)$ and special unipotent classes of $\G$. It would be natural to use those special unipotent classes to parametrise the unipotent characters of $G$. For that purpose Lusztig introduced the following finite group attached to the family.

\begin{Def}\label{sec:canonical-quotient}
  \item Let $u \in \G$ be a special unipotent element, and $b_u$ be
    the dimension of the variety of Borel subgroups of $\G$ containing
    $u$. Let $\mathcal{S} \subset \Irr_K(A_{\G}(u))$ be defined by
$$\mathcal{S}:= \{\psi \in \Irr_K(A_{\G}(u))| \quad a_{\phi_{u, \psi}}=b_u \}$$
The \textit{canonical quotient} $\Gamma_u$ is the smallest quotient of $A_{\G}(u)$ through which every element of $\mathcal{S}$ factors. In other words $$\Gamma_u := A_{\G}(u) / \bigcap_{\psi \in \mathcal{S}} \ker \psi.$$
\end{Def}

\begin{Rk}
   If $C_u$ is an $F$-stable class, $F$ acts naturally on $A_\G(u)$ so we can define $\tilde{\mathcal{M}}_{\ell}(\Gamma^\ell_u)$ and according to \cite[13.1.3]{lusztig_g._characters_1984}, $\tilde{\mathcal{M}}_{\ell}(\Gamma^\ell_u)$ only depend on the $\G$-conjugacy class of $u$.
   \end{Rk}

\vspace*{0.1cm}

Lusztig showed in \cite[13.1.3]{lusztig_g._characters_1984} that canonical quotients are exactly the finite groups which appear in the classification of unipotent characters:

 \begin{Pro}
 Let $u$ be a special unipotent element of $\G$,  and $\mathcal{F}$ be the family of $\Irr_K(W)$ containing the special character $ \phi_{u,1} \otimes \varepsilon $ where $\varepsilon$ is the sign character. Then, the triple  $\Gamma_u \subset \tilde{\Gamma_u} \supset \Gamma_u.F$ is isomorphic to $\mathcal{G}_{\mathcal{F}} \subset \tilde{\mathcal{G}}_{\mathcal{F}} \supset \mathcal{G}_{\mathcal{F}}$.
 \end{Pro}

% \vspace*{0.5cm}

Hence, there is a parametrisation of unipotent characters in term of special unipotent classes:

 \begin{The}\label{sec:param-unip}
The unipotent characters of $G$ are parametrised by pairs $(u,x)$ where $u \in G$ is a special unipotent element up to $\G$-conjugacy and $x \in \tilde{\mathcal{M}}(\Gamma_u)$.
\end{The}

\subsection{Another way of counting irreducible Brauer characters of unipotent blocks}

\subsubsection{$\ell$-special classes}
Can we adapt the result of the previous section to the positive characteristic framework? More precisely, can we count the number of irreducible modular characters of $B_1$ by using information coming from unipotent classes? When $\ell$ is good, the fact that unipotent characters of $G$ form a basic set for the unipotent blocks together with Corollary \ref{sec:param-unip} give a positive answer to this question. When $\ell$ is bad, it is not true anymore that unipotent characters form a basic set as seen in Remark \ref{bad-examples}. So we need to adapt some definitions from the previous section. For that purpose, recall that  given $W'$ a reflection subgroup of $W$, there exists a map
$$j : \Irr(W') \longrightarrow \Irr(W)$$
called the \textit{$j$-induction} with the following property: if $W'$ is a parabolic subgroup, $j$ maps any special character of $W'$ to a special character of $W$. By \cite[13.3]{lusztig_g._characters_1984}, the $j$-induction is involved in the construction of  a map defined by Lusztig sending some special conjugacy classes of $\G^*$ onto unipotent classes of $\G$. We say that theconjugacy class $C_g$ of $g$ in $\G^*$ is special if $g$ has Jordan decomposition $g=su$ and $u$ is special in $C_{\G^*}(s)$ (notice that $C_{\G^*}(s)$ is connected since $\G$ is adjoint). %{\bf[Attention tu as utilise $C_v$ et pas $(v)$ auparavant pour les classes de conjugaison]}.

\smallskip

In addition, in \cite[13.3]{lusztig_g._characters_1984}, Lusztig introduced a map $\Phi$ between special classes of $\G^*$ and unipotent classes of $\G$ defined as follows. Let $g$ be a special element of $\G^*$ with Jordan decomposition $g=su$. The Springer correspondence affords a special irreducible representation $E_{u,1}$ of the Weyl group $W_s^*$ of $C_{\G}^*(s)$. Using the $j$-operation and the natural isomorphism $W^* \simeq W$, we get an irreducible representation of $W$. Finally, the Springer correspondence provides us a unipotent element $v$ of $\G$, such that $E=E_{v,1}$. Then $\Phi$ sends the $\G$-conjugacy class of $g$ to the $\G$-conjugacy class of $v$. 

\smallskip

We have everything we need to define an ``$\ell$-modular version'' of special unipotent classes:

%\vspace*{0.3cm}

\begin{Def}
A unipotent element $u$  of $\G$ is \textit{$\ell$-special} if there is a special element $g \in \G^*$ such that $\Phi$ associates the class of $g$ to the class of $u$ and the semisimple part of $g$ is an isolated $\ell$-element of $\G^*$.
\end{Def}

%\vspace*{0.3cm}

\begin{Rk} The following properties can be readily deduced from the definition:
  \begin{itemize}
  \item[(i)] Every special element is the image of a special class of $\G^*$, therefore every special element is $\ell$-special.
  \item[(ii)] If $\ell$ is good, then there is no non-trivial isolated semi simple $\ell$-element, so the $\ell$-special unipotent elements are exactly the special ones.
  \item[(iii)] When $\G$ is of type $B$, $C$ or $D$, the prime number $\ell = 2$ is the only bad prime so any unipotent element  is $2$-special.
  \end{itemize}
\end{Rk}

\vspace*{0.2cm}

In the characteristic $0$ case, to count unipotent characters, we associate to each $F$-stable special class $C_u$ of $\G$, the canonical quotient $\Gamma_u$. In the positive characteristic case, this group appears to be too small to count the irreducible modular characters of the unipotent blocks. We shall instead consider the following generalisation.

%\vspace*{0.3cm}

\begin{Def}\label{sec:ell-special-quo}
  \leavevmode
  \begin{enumerate}[label=\arabic*)]
  \item Let $u$ be an $\ell$-special unipotent element of $\G$. If $\Psi$ is a projective character of $A_{\G}(u)$ we define $a_{\Psi}$ as:
    $$a_{\Psi}:=\min a_{\phi_{u, \psi}}$$
    where $\psi$ runs over the irreducible components of $\Psi$ such that $\phi_{u, \psi} \neq 0$. Let $\mathcal{S}_{\ell}$ be the set of indecomposable projectives $\Psi$ such that $a_{\Psi}$ is maximal. Then the \textit{$\ell$-special quotient}  $\Gamma^\ell_u$ is the smallest quotient of $A_\G(u)$ such that every $\Psi \in \mathcal{S}_{\ell}$ factors through $\Gamma^\ell_u$.
    \item Let $\Gamma$ be a finite group on which $F$ acts. We denote by $\tilde{\Gamma}$ the semidirect product of $\Gamma$ with the infinite cyclic group generated by $F$ such that $FxF^{-1}=F(x)$ if $x \in \Gamma$. $\tilde{\mathcal{M}}_{\ell}(\Gamma)$ is the set of $\tilde{\Gamma}$-conjugacy class of pairs $(x, \phi)$ where $x$ is an element of $\Gamma.F$ and $\phi \in \Irr_k(C_{\Gamma}(x))$.
    \end{enumerate}
    \end{Def}

\begin{Rk}\label{rmk:char0}
  Unlike the characteristic zero case, it is important to consider representations over $k$ (a field of characteristic $\ell$) to define the set $\mathcal{M}_\ell$ when $\ell$ is a bad prime number. However, when $\ell$ is good and does not divide $|Z(\G)/Z(\G)^\circ|$, the groups $A_\G(u)$ are $\ell'$-groups so that in this case $\Gamma^\ell_u = \Gamma_u$ and $\tilde{\mathcal{M}}_{\ell}(\Gamma) \simeq \tilde{\mathcal{M}}(\Gamma)$.

\end{Rk}

\vspace*{0.3cm}

If $C_u$ is an $\ell$-special $F$-stable unipotent class, we denote by $\alpha_{\ell,u}$ the cardinal of $\tilde{\mathcal{M}}_\ell(\Gamma^{\ell}_u)$. Let $\alpha_\ell$ (or $\alpha_{\ell}(G)$ if there is an ambiguity in the underlying group) be the number of pairs $(u, x)$ where $u$ is an $\ell$-special unipotent element of $G$ up to $\G$-conjugacy and $x \in \tilde{\mathcal{M}}_\ell(\Gamma^{\ell}_u)$.  We conjecture the following result:

\begin{Con}
  Suppose that $p$ is good for $\G$. Then, $\alpha_\ell$ is the number of unipotent irreducible Brauer characters of $G$.
\end{Con}

 By Remark \ref{rmk:char0}, this holds when $\ell$ is good and does not divide $|Z(\G)/Z(\G)^\circ|$ since in this case the unipotent characters form a basic set. In the case of bad characteristic, the value of $m_1$ was computed in \cite{geck_modular_1997} for adjoint simple groups, so we simply need to compute $\alpha_\ell$ to check that our conjecture holds. To be able to do this, we need to know how to detect $\ell$-special classes of a group and for each such class, how to compute $\ell$-special quotients.

 \subsection{Computing $\alpha_\ell$}

 Let us begin by the following observation: according to \ref{sec:MGamma}, if $u$ is an $\ell$-special unipotent element of $G$ such that $F$ acts trivially on $A_G(u)$, to compute $\alpha_{\ell, u}$, it is enough to count the number of conjugacy classes of pairs $(x, \sigma)$ where $x \in \Gamma^{\ell}_u$ and $\sigma \in \Irr_k(C_{\Gamma}(x))$. Moreover, by \cite[2.4]{taylor_unipotent_2013}, for adjoint exceptionnal groups, every $F$-stable unipotent class has a representant $u$ such that $F$ acts trivially on $A_G(u)$. That observation will be used later.

 \vspace{0.3cm}
 
    \subsubsection{Classical type} %{\bf[Ca devrait aussi marcher pour spin et half spin mais laissons ca comme ca pour l'instant]}
    Suppose $\ell=2$ and $\G$ is simple of type $B$, $C$ or $D$. Then for every unipotent element $u$ of $G$, $u$ is $2$-special, $A_\G(u)$ is a $2$-group  and in particular, the only indecomposable projective $kA_\G(u)$-module (resp. irreducible $kA_\G(u)$-module) is the regular representation (resp. the trivial representation). This shows that $\Gamma_u^2 = A_\G(u)$ and, by \ref{sec:MGamma},  $\tilde{\mathcal{M}}_2(\Gamma^2_u)$ corresponds bijectively to the set of $F$-conjugacy classes of $A_{\G}(u)$. By \cite[3.25]{digne_representations_1991}, the $F$-conjugacy classes of $A_\G(u)$ parametrize the $G$-orbits in $C_u^F$. Summing over all unipotent $\G$-conjugacy classes, this shows that $\alpha_2$ equals the number of unipotent $G$-conjugacy classes, which by Proposition \ref{sec:countingirrkb1} equals $m_1(G)$.

    \begin{Rk}
If $\G$ is neither spin or half-spin, then, by looking at the proof of \cite[2.4]{taylor_unipotent_2013}, for every $F$-stable unipotent class of $\G$ there is a representant $u$ such that $F$ acts trivially on $A_G(u)$. Hence, it is enough to work with conjugacy classes of $A_G(u)$.
\end{Rk}

    \vspace{0.3cm}
     
\subsubsection{Exceptional groups}

For simple adjoint groups of exceptional type, the number of irreducible modular characters lying in the unipotent blocks has been determined by Geck--Hiss (see \cite[6.6]{geck_modular_1997}).

$$\begin{array}{|c|cccc|}
  \hline
  \text{Type} & \ell=2 & \ell=3 & \ell=5 & \ell \text{ good}\\
    \hline
  G_2 & 8 & 9 & & 10\\
  \hline
  F_4 & 28 & 35 & & 37\\
  \hline
  E_6 & 27 & 28 & & 30 \\
  \hline
  E_7 & 64 & 72 & & 76 \\
  \hline
  E_8 & 131 & 150 & 162 & 166\\
  \hline
\end{array}$$

\vspace{0.3cm}

    Using this table, we will show that for any exceptional adjoint group $G$ and any bad prime number $\ell$ we have $\alpha_{\ell}(G)=m_1(G)$. Let $S_W$ be the set of special characters of $W$ and let $\bar{S}_W$ be the image of the injective map $C_u \mapsto E_{u,1}$ from the set of unipotent classes of $\G$ to Irr$_K(W)$. We have $S_W \subset \bar{S}_W$. There is a nice description of the set $\bar{S}_W$ in term of $j$-induction of characters of maximal subgroups of $W$. More precisely, let us fix a set of roots $\Phi$ and a set of simple roots $\Delta= \{\alpha_1,\ldots, \alpha_n \}$. For $\alpha \in \Phi$, let $s_{\alpha} \in W$ be the corresponding reflection. We have $W=\langle s_{\alpha_i},1 \leq i \leq n \rangle$. Let $\alpha_0=\sum n_{\alpha_i}\alpha_i$ be the highest short root. For $i \in 0, \ldots, n$, let
    $$W_i=\langle s_{\alpha_j},0 \leq j \neq i \leq n \rangle$$
It was shown by Shoji (\cite{shoji_conjugacy_1974}, \cite{shoji_springer_1979}), Springer (\cite{springer_trigonometric_1976}), Alvis and Lusztig (\cite{alvis_springers_1982}) that $\bar{S}_W$ is exactly the set of irreducible characters of $W$ coming from $j$-induction of some special character of some $W_i$. This fact leads to the following definition.

%\vspace*{0.3cm} 

\begin{Def}
  Let $\ell$ be a prime number, $E \in \bar{S}_W$. We say that $E$ is \textit{$\ell$-special} if there is $i \in  1, \dots, n $ such that:
  \begin{itemize}    
  \item $E$ is coming from the $j$-induction of a special character of $W_i$
  \item $n_{\alpha_i}$ is a power of $\ell$
  \end{itemize}
\end{Def}

\vspace*{0.3cm} 

The classification of quasi-isolated semisimple elements by Bonnaf\'e in \cite{bonnafe_quasi-isolated_2005}  shows that every group of the form $W_i$ is the Weyl group of the centraliser of an isolated semisimple element $s_i$ of $G^*$ whose order is equal to $n_{\alpha_i}$. It turns out that $u$ is $\ell$-special if and only if $E_{u,1}$ is $\ell$-special, which justifies our definition. Using this property, we can sketch a strategy to compute $\alpha_{\ell}$.
% Moreover, by [ref] for any unipotent classes $\mathcal{O}$ , we can chose a representatives $u$ in $\mathcal{O}^F$ such that $F$ acts trivially on $A_{\G}(u)$, in particular $F$-conjugacy classes are simply conjugacy classes and if $x \in \Gamma^\ell_u$, $C_{\Gamma^\ell_u,F}(x)=C_{\Gamma_{^\ell_u}}(x)$. Using those facts, we can sketch a general strategy to compute $\alpha_\ell$.

\medskip
\noindent {\bf First step.} Finding every $\ell$-special unipotent classes:
  \begin{enumerate}
  \item Determine all maximal subgroups $W_i^*$ of $W^*$ the Weyl group of $\G^*$;
  \item By using $j$-induction tables, find every $\ell$-special irreducible characters of $W^*$;
  \item For every $\ell$-special character $E$, use the Springer correspondence to get the unipotent class $C_u$ of $\G$ such that $E=E_{u,1}$ where $E$ is viewed as a representation of $W$ via the natural isomorphism $W \simeq W^*$.
  \end{enumerate}
\noindent{\bf Second step.} Now that we have every $\ell$-special class, compute $\alpha_{\ell}$:
  \begin{enumerate}
  \item For each $F$-stable $\ell$-special class $C_u$, compute the
    $\ell$-special quotient $\Gamma^{\ell}_u$ of $A_G(u)$;
  \item For each conjugacy class $(x)$ of $\Gamma^\ell_u$, compute $|\Irr_k(C_{\Gamma^\ell_u}(x))$| by counting the number of $\ell'$-classes. Let $\alpha_{\ell, u}$ be the sum over conjugacy classes of the numbers obtained this way;
  \item Then $\alpha_\ell$ is equal to $\sum_u \alpha_{\ell, u}$ where $u$ runs over the set of $\ell$-special classes.
  \end{enumerate}

 \smallskip
 
 Calculations were made for adjoint exceptional groups using CHEVIE (\cite{michel_development_2015}). For each group of exceptional type $\G$ and each bad prime $\ell$, we provide in the appendix tables listing $\ell$-special classes of $\G$ (we use the same labelling as in CHEVIE) and for each special class $C_u$ the groups $A_{\G}(u)$, $\Gamma^\ell_u$ and the number $\alpha_{\ell,u}$.

\vspace{0.3cm}
 
\subsubsection{ $A_n$ type}

Suppose that $\G$ is $\mathrm{SL}_n(\overline{\F}_p)$, in which case $G$ is either $\mathrm{SL}_n(q)$ or $\mathrm{SU}_n(q)$. By \cite{kleshchev_representations_2009} and \cite{denoncin_stable_2017} we know that $G$ has a unitriangular basic set, and by carefully studying how this basic set has been obtained, we will show that $\alpha_{\ell}$ is equal to $m_1$ in this case. Recall that unipotent classes of $\G$ are parametrised by partitions of $n$. We denote by $C_{\lambda}$ the unipotent class corresponding to the partition $\lambda=(\lambda_1,\ldots,\lambda_r)$ and by $u_{\lambda} \in C_{\lambda}^F$  a rational representative in this class. Then $A_{\G}(u_{\lambda})$ is a cyclic group of order $m_{\lambda}:=\gcd(\lambda_1,\ldots,\lambda_r)_{p'}$.

\smallskip

Let us compute $\Gamma^{\ell}_{u_{\lambda}}$ and then $\alpha_{\ell,u_{\lambda}}$. With the notation of \ref{sec:ell-special-quo}, recall that $\Gamma^{\ell}_{u_{\lambda}}$ is the smallest quotient of $A_{\G}(u_{\lambda})$ through which every projective character in $\mathcal{S}_{\ell}$ factors. Since the only representation of $A_{\G}(u_{\lambda})$ whose Springer correspondent is non-zero is the trivial representation, the only projective character in $\mathcal{S}_\ell$ is the sum of the irreducible characters whose image by the decomposition map are trivial (namely the $\ell$-characters of the cyclic group $A_{\G}(u_{\lambda})$). The kernel of this projective character is the cyclic subgroup of $A_{\G}(u_{\lambda})$ of order $(m_{\lambda})_{\ell'}$, so $\Gamma^{\ell}_{u_{\lambda}}$ is a cyclic group of order $(m_{\lambda})_{\ell}$. In particular, $\Gamma^{\ell}_u$ is a $\ell$-group so for any $x \in \tilde{\Gamma}^{\ell}_u$, $|\Irr_k(C_{\Gamma^{\ell}_{u_{\lambda}}}(x))|=1$. Hence, according to \ref{sec:MGamma}, $\alpha_{\ell, u_{\lambda}}$ is the number of $F$-classes of $\Gamma^{\ell}_{u_{\lambda}}$, that is $\gcd(m_{\lambda},q+ \varepsilon)_{\ell}$ where $\varepsilon=1$ if $G=SL_n(q)$ and $\varepsilon=-1$ if $G=SU_n(q)$.

\smallskip

Let us explain briefly how the basic set mentioned above is obtained: let $\widetilde{\G}:=\mathrm{GL}_n(\overline{\F}_p)$, and $\widetilde{G}$ be $\mathrm{GL}_n(q)$ or $\mathrm{GU}_n(q)$. By \cite{dipper_decomposition_1985} and \cite{geck_decomposition_1991}, the unipotent characters  form a unitriangular basic set for the unipotent blocks of $\widetilde{G}$. By slightly changing  this basic set, we get a basic set $\mathcal{B}$ for the unipotent block of $G$ with the following properties (see \cite{denoncin_stable_2017} for more details):

\begin{itemize}
\item For each unipotent classes $C_{\lambda}$ of $\widetilde{\G}$, there is a unique element $\rho_{\lambda}$ of $\mathcal{B}$ whose wave front set is $C_{\lambda}$;
\item The restriction of $\rho_{\lambda}$ to $G$ has no multiplicity and has $\alpha_{\ell,u_{\lambda}}$ irreducible constituents (using the previous computation for $\alpha_{\ell, u_{\lambda}}$).
\end{itemize}
This shows that $m_1(G)=\alpha_{\ell}$.

\vspace*{0.5cm}

\subsubsection{Conclusion}

We state here a theorem compiling the results obtained above.
\begin{The}
  Suppose that $p$ is good for $\G$ and that we are in one of the following cases
\begin{itemize}
\item $\G$ is simple of type $B$, $C$ or $D$.
\item $\G$ is an exceptional group of adjoint type.
\item $\G$ is $\mathrm{SL}_n(\overline{\F}_p)$ (and therefore $G$ is either $\mathrm{SL}_n(q)$ or $\mathrm{SU}_n(q)$).
\end{itemize}  
 Then the number $\alpha_\ell$ of  pairs $(u, x)$ where $u$ is an $\ell$-special unipotent element of $\G$ up to $\G$-conjugacy and $x \in \mathcal{M}_\ell(\Gamma^{\ell}_u)$ is equal to $m_1$, the number of unipotent irreducible Brauer characters.
\end{The}

\vspace*{0.5cm}

\section*{Acknowledgements}
I would like to thank my PhD advisors Olivier Brunat and Olivier Dudas for their guidance and their encouragement. I would also like to show my gratitude to Prof. Gunter Malle for his useful comments.

  \newpage

%\newcolumntype{C}[1]{>{\centering\\\arraybackslash\vspace{2pt}}m{#1}}
\newcolumntype{C}[1]{>{\centering\arraybackslash}p{#1}}

\section*{Appendix: $\ell$-special unipotent classes for exceptional groups ($\ell$ bad)}

Here $A_G(u)$ denote the group of connectod component of the class, $\Gamma^\ell_u$ denote the $\ell$-special quotientand $\alpha{_\ell,u}$ the cardinal of $\tilde{\mathcal{M}}_\ell(\Gamma^\ell_u)$ (see Definition \ref{sec:ell-special-quo}). Having blank entries for $\Gamma^\ell_u$ and $\alpha_{\ell,u}$ means that the corresponding class is not $\ell$-special.

\subsection*{Type $G_2$}

$$\begin{array}{|c|c|cc|cc|}
  \hline
     \text{Classes} & A_G(u) & \Gamma^2_u & \alpha_{2,u} & \Gamma^3_u & \alpha_{3,u}  \\
     \hline
    1 & 1 & 1 & 1 & 1 & 1\\
  A_1 & 1 &  &  & 1 & 1  \\  
  \widetilde{A}_1 & 1 & 1  & 1 &  &   \\
    G_2(a_1) & S_3 & S_3 & 6 & S_3 & 5 \\
     G_2 & 1 & 1 & 1  & 1 & 1\\
     \hline
   \end{array}$$

\subsection*{Type $F_4$}
\begin{center}
$  \begin{array}{|c|c|cc|cc|}
  \hline
  \text{Classes} & A_G(u) & \Gamma^2_u & \alpha_{2,u}  & \Gamma^3_u & \alpha_{3,u} \\
  \hline
    1 & 1 & 1 & 1 & 1 & 1 \\
    
    A_1  & 1 & 1 & 1  & & \\
         
  \tilde{A_1} & S_2 & S_2 & 2 & S_2 & 4 \\
  
    A_1+\tilde{A_1} & 1 & 1 & 1 & 1 & 1 \\
  
    \tilde{A_2} & 1 & 1 & 1 &1 & 1\\
  
    A_2 & S_2 & S_2 & 2 & 1 & 1\\
    
     A_2+\tilde{A_1} & 1 & 1 & 1 & & \\
     
    \tilde{A_2}+A_1 &1 & &  & 1 & 1 \\
    
    B_2 & S_2 & S_2 & 2 & & \\
    
    C_3(a_1) & S_1 & S_1 & 2 & &\\ 
    
    F_4(a_3) & S_4 & S_4 & 8 & S_4 & 18\\
    
    C_3 & 1 & 1 & 1 & 1 & 1 \\
    
    B_3 & 1 & 1 & 1 & 1 & 1\\
    
    F_4(a_2) & S_2 & S_2 & 2 & 1 & 1 \\
     
    F_4(a_1) & S_2 & S_2 & 2 & S_2 & 4\\
    
    F_4 & 1 & 1 & 1 & 1 & 1 \\
    
    \hline
   \end{array} $
 \end{center}

\pagebreak

\subsection*{Type $E_6$}

 \begin{center}
  $\begin{array}{|c|c|cc|cc|}
  \hline
  \text{Classes} & A_G(u) & \Gamma^2_u & \alpha_{2,u} & \Gamma^3_u & \alpha_{3,u} \\
  \hline
    1 & 1 & 1 & 1 & 1 &1  \\
  
  A_1 & 1 & 1 & 1 & 1 & 1\\
    
    2A_1 & 1 & 1 & 1 & 1 & 1\\
    
    3A_1 & 1 & 1 & 1 & & \\
    
    A_2 & S_2 & S_2 & 2 & S_2 & 4\\
    
    A_2+A_1 & 1 & 1 & 1 & 1 & 1\\
    
    A_2+2A_1 & 1 & 1 & 1 & 1 & \\
    
    2A_2 & 1 & 1 & 1 & 1 & 1 \\
     
    2A_2+A_1 & 1 & & & 1 & 1 \\

    A_3 & 1 & 1 & 1 & 1 & 1 \\
    
    A_3+A_1 & 1 & 1 & 1 & & \\
    
    D_4(a_1) & S_3 & S_3 & 6 & S_3 & 5 \\
    
    A_4 & 1 & 1 & 1 & 1 & 1 \\
    
    D_4 & 1 & 1 & 1 & 1 & 1 \\
    
    A_4+A_1 & 1 & 1 & 1 & 1 &1 \\
    
    D_5(a_1) & 1 & 1 & 1 & 1 & 1 \\
    
    A_5 & 1 & 1 & 1 & & \\
    
    E_6(a_3) & S_2 & S_2 & 2 & S_2 & 4 \\
    
    D_5 & 1 & 1 & 1 & 1 & 1 \\
    
    E_6(a_1) & 1 & 1 & 1 & 1 & 1 \\
    
    E_6 & 1 & 1 & 1 & 1 & 1\\
    \hline
\end{array}$
\end{center}

\pagebreak

  \subsection*{Type $E_7$}

\begin{center}
  $\begin{array}{|c|c|cc|cc|}
  \hline
  \text{Classes} & A_G(u) & \Gamma^2_u & \alpha_{2,u} & \Gamma^3_u & \alpha_{3,u} \\
  \hline
    1 & 1 & 1 & 1 & 1 & 1 \\
  
    A_1 & 1 & 1 & 1 & 1 & 1\\
  
    2A_1 & 1 & 1 & 1 & 1 & 1 \\
  
    3A_1' & 1 & 1 & 1 &  &  \\
    
    3A_1'' & 1 & 1 & 1 & 1 & 1 \\

    A_2 & S_2 & S_2 & 2 & S_2 & 4 \\
    
    4A_1 & 1 & 1 & 1 & & \\
    
    A_2+A_1 & S_2 & S_2 & 2 & S_2 & 4 \\
    
    A_2+2A_1 & 1 & 1 & 1 & 1 & 1 \\
    
    A_2+3A_1 & 1 & 1 & 1 & 1 & 1 \\
    
    A_3 & 1 & 1 & 1 & 1 & 1 \\
    
    2A_2 & 1 & 1 & 1 & 1 & 1 \\
     
    2A_2 + A_1 & 1 & & 1 & 1 & 1 \\
    
    (A_3+A_1)'' & 1 & 1 & 1 & 1 & 1 \\
    
    (A_3+A_1)' & 1 & 1 & 1 & & \\
    
    A_3+2A1 & 1 & 1 & 1 & & \\
    
    D_4(a_1) & S_3 & S_3 & 6 & S_3 & 5 \\
     
    D_4(a_1)+A_1 & S_2 & S_2 & 2 & S_2 & 4\\
    
    A_3+A_2 & S_2 & S_2 & 2 & 1 & 1\\
    
    A_4 & S_2 & S_2 & 2 & S_2 & 4\\
    
    A_3+A_2+A_1 & 1 & 1 & 1 & 1 & 1 \\
    
    D_4 & 1 & 1 & 1 & 1 & 1 \\
    
    A_4+A_1 & S_2 & S_2 & 2 & S_2 & 4 \\
    
    D_4+A_1 & 1 & 1 & 1 &  &  \\
    
    D_5(a_1) & S_2 & S_2 & 2 & S_2 & 4\\
    
    A_5'' & 1 & 1 & 1 & 1 & 1\\
    
    A_4+A_2 & 1 & 1 & 1 & 1 & 1\\
     
    A_5+A_1 & 1 & & 1 & 1 & 1\\
    
    D_5(a_1)+A_1 & 1 & 1 & 1 & 1 & 1 \\
    
    A_5' & 1 & 1 & 1 & & \\
    
    E_6(a_3) & S_2 & S_2 & 2 & S_2 & 4 \\
    
    D_6(a_2) & 1 & 1 & 1 & & \\
    
    D_5 & 1 & 1 & 1 & 1 & 1\\
    
    E_7(a_5) & S_3 & S_3 & 6 & S_3 & 5 \\
    
    D_6(a_1) & 1 & 1 & 1 & 1 & 1\\
    
    D_5+A_1 & 1 & 1 & 1 & 1 & 1\\
    
    A_6 & 1 & 1 & 1 & 1 & 1\\
    
    E_7(a_4) & S_2 & S_2 & 2 & 1 & 1\\
    
    D_6 & 1 & 1 & 1 & & \\
    
    E_6(a_1) & S_2 & S_2 & 2 & S_2 & 4 \\
    
    E_7(a_3) & S_2 & S_2 & 2 & S_2 & 4 \\
    
    E_6 & 1 & 1 & 1 & 1 & 1\\
    
    E_7(a_2) & 1 & 1 & 1 & 1 & 1\\
    
    E_7(a_1) & 1 & 1 & 1 & 1 & 1\\
    
    E_7 & 1 & 1 & 1 & 1 & 1\\
    \hline
\end{array}$
\end{center}

\pagebreak

  \subsection*{Type $E_8$}

\begin{center}
 $\begin{array}{|c|c|cc|cc|cc|}
  \hline
  \text{Classes} & A_G(u) & \Gamma^2_u & \alpha_{2,u} & \Gamma^3_u & \alpha_{3,u} & \Gamma^5_u & \alpha_{5,u} \\
  \hline
    1 & 1 & 1 & 1 & 1 & 1 & 1 & 1 \\
  
  A_1 & 1 & 1 & 1 & 1 & 1 & 1 & 1 \\
  
    2A_1 & 1 & 1 & 1 & 1 & 1 & 1 & 1 \\
  
  3A_1 & 1 & 1 & 1 & & & &  \\
    
    A_2 & S_2 & S_2 & 2 & S_2 & 4 & S_2 & 4 \\
    
    4A_1 & 1 & 1 & 1 & & & & \\
    
    A_2+A_1 & S_2 & S_2 & 2 & S_2 & 4 & S_2 & 4 \\
    
    A_2+2A_1 & 1 & 1 & 1 & 1 & 1 & 1 & 1 \\
    
    A_2{+}3A_1 & 1 & 1 & 1 & & & & \\
    
    2A_2 & 2 & 2 & 2 & S_2 & 4 & S_2 & 4\\
    
    A_3 & 1 & 1 & 1 & 1 & 1 & 1 & 1 \\

    2A_2{+}A_1 & 1 & & & 1 & 1 & & \\
    
    A_3{+}A_1 & 1 & 1 & 1 & & & & \\

    2A_2{+}2A_1 & 1 & & 1 & 1 & 1 & & \\
    
    D_4(a_1) & S_3 & S_3 & 6 & S_3 & 5 & S_3 & 8\\
    
    A_3+2A_1 & 1 & 1 & 1 & & & & \\
    
    D_4(a_1)+A_1 & S_3 & S_3 & 6 & S_3 & 5 & S_3 & 8 \\
    
    A_3{+}A_2 & S_2 & S_2 & 2 & S_2 & 1 & 1 & 1\\
    
    A_3{+}A_2{+}A_1 & 1 & 1 & 1 & & & &  \\
    
    D_4(a_1){+}A_2 & S_2 & S_2 & 2 & S_2 & 4 & S_2 & 4 \\
    
    A_4 & S_2 & S_2 & 2 & S_2 & 4 & S_2 & 4 \\
    
    D_4 & 1 & 1 & 1 & 1 & 1 & 1 & 1 \\
    
    2A_3 &  1 & 1 &  1 & & & & \\
     
    A_4{+}A_1 & S_2 & S_2 & 2 & S_2 & 4 & S_2 & 4 \\
       
    D_4{+}A_1 & 1 &  1 &  1 & & & & \\
        
    A_4{+}2A_1 & S_2 & S_2 & 2 & S_2 & 4 & S_2 & 4 \\
    
    A_4{+}A_2 & 1 & 1 & 1 & 1 & 1 & 1 & 1\\
    
    D_5(a_1) & S_2 & S_2 & 2 & S_2 & 4 & S_2 & 4\\
    
    D_5(a_1){+}A_1 & 1 & 1 & 1 & 1 & 1 & 1 & 1\\
    
    A_4{+}A_2{+}A_1 & 1 & 1 & 1 & 1 & 1 & 1 & 1\\
    
    A_5 & 1 & 1 & 1  & & & & \\
    
    D_4{+}A_2 & S_2 & S_2 & 2 & 1 & 1 & 1 & 1\\

    A_4{+}A_3 & 1 & & & & & 1 & 1 \\
    
    D_5(a_1){+}A_2 & 1 & 1 & 1 & & & &  \\

    A_5{+}A_1 & 1 & & & & & & \\
    
    E_6(a_3) & S_2 & S_2 & 2 & & &S_2 & 4 \\

    D_6(a_2) & S_2 & S_2 & 2 & & & & \\
    
    E_6(a_3){+}A_1 & S_2 & & S_2 & 4 & & &  \\
    
    D_5 & 1 & 1 & 1 & 1 & 1 & 1 & 1 \\
    
    E_7(a_5) & S_3 & S_3 & 6 & & & & \\
    
    D_5{+}A_1 & 1 & 1 & 1 & & & & \\
    
    E_8(a_7) & S_5 & S_5 & 18 & S_5 & 27 & S_5 & 34\\
    
    A_6 & 1 & 1 & 1 & 1 & 1 & 1 & 1 \\
    
    D_6(a_1) & S_2 & S_2 & 2 & S_2 & 4 & S_2 & 4 \\
     
    A_6{+}A_1 & 1 & 1 & 1 & 1 & 1 & 1 & 1 \\
    
    E_7(a_4) & S_2 & S_2 & 2 & 1 & 1 & 1 & 1 \\

\hline
        \end{array}$

      \end{center}

      \pagebreak

      Type $E_8$ -continued

\vspace*{0.1cm}
      
 \begin{center}
  $\begin{array}{|c|c|cc|cc|cc|}
    \hline
  \text{Classes} & A_G(u) & \Gamma^2_u & \alpha_{2,u} & \Gamma^3_u & \alpha_{3,u} & \Gamma^5_u & \alpha_{5,u} \\
    \hline

    D_5{+}A_2 & S_2 & S_2 & 2 & 1 & 1 & 1 & 1 \\
    
    E_6(a_1) & S_2 & S_2 & 2 & S_2 & 4 & S_2 & 4 \\

    D_7(a_2) & S_2 & S_2 & 2 & S_2 & 4 & S_2 & 4 \\
    
    A_7 & 1 & 1 & 1 & & & & \\
    
    E_6(a_1){+}A_1 & S_2 & S_2 & 2 & S_2 & 4 & S_2 & 4 \\ 

    D_6 & 1 & 1 & 1 & & & & \\
      
    E_8(b_6) & S_3 & S_3 & 2 & S_3 & 5 & S_2 & 4 \\
    
    E_7(a_3) & S_2 & S_2 & 2 & S_2 & 4 & S_2 & 4\\
    
    E_6 & 1 & 1 & 1 & 1 & 1 & 1 & 1 \\
    
     D_7(a_1) & S_2 & S_2 & 2 & 1 & 1 & 1 & 1 \\
     
     E_6+A_1 & 1 & & & 1 & 1 & & \\
    
    E_8(a_6)  & S_3 & S_3 & 6 & S_3 & 5 & S_3 & 8 \\
    
    E_7(a_2) & 1 & 1 & 1 & & & & \\
    
    D_7 & 1 & 1 & 1 & & & & \\
    
    E_8(b_5) & S_3 & S_3 & 6 & S_3 & 5 & S_3 & 8 \\
    
    E_8(a_5) & S_2 & S_2 & 2 & S_2 & 4 & S_2 & 4\\

    E_7(a_1) & 1 & 1 & 1 & 1 & 1 & 1 & 1 \\
       
    E_8(b_4) & S_2 & S_2 & 2 & 1 & 1 & 1 & 1 \\
    
    E_8(a_4) & S_2 & S_2 & 2 & S_2 & 4 & S_2 & 4 \\
    
    E_7 & 1 & 1 & 1 & & & & \\
    
    E_8(a_3) & S_2 & S_2 & 2 & S_2 & 4 & S_2 & 4\\
    
    E_8(a_2) & 1 & 1 & 1 & 1 & 1 & 1 & 1 \\
    
    E_8(a_1) & 1 & 1 & 1 & 1 & 1 & 1 & 1 \\
        
    E_8 & 1 & 1 & 1 & 1 & 1 & 1 & 1 \\  
     \hline
\end{array}$
\end{center}

\newpage

\bibliography{bib}{}
\bibliographystyle{alpha}

\end{document}